\documentclass[11pt,reqno]{amsart}

\usepackage{amsbsy,amsfonts,amsmath,amssymb,amscd,amsthm}
\usepackage{mathrsfs,float}
\usepackage[mathcal]{euscript}
\usepackage{subfigure,enumitem,graphicx,epstopdf}
\usepackage[a4paper,text={6.1in,8in},centering,includefoot,foot=1in]{geometry}
\usepackage{pxfonts,epstopdf,color}
\usepackage{srcltx}
\usepackage[margin=20pt,font=small,labelfont=bf,textfont=it]{caption}

\small\normalsize
\theoremstyle{plain}
\newtheorem{maintheorem}{Theorem}

\newtheorem{theorem}{Theorem}[section]
\newtheorem{lemma}[theorem]{Lemma}
\newtheorem{proposition}[theorem]{Proposition}
\newtheorem{corollary}[theorem]{Corollary}

\newtheorem{definition}[theorem]{Definition}
\newtheorem{remark}[theorem]{Remark}

\theoremstyle{definition}
\newtheorem{example}[theorem]{Example}

\newtheorem{claim}[theorem]{Claim}

\DeclareMathOperator{\diam}{diam}
\newcommand{\eqdef}{\stackrel{\scriptscriptstyle\rm def}{=}}

\let\oldtocsection=\tocsection
\let\oldtocsubsection=\tocsubsection
\renewcommand{\tocsection}[2]{\hspace{0em}\bf\oldtocsection{#1}{#2}}
\renewcommand{\tocsubsection}[2]{\hspace{1em}\oldtocsubsection{#1}{#2}}
\let\oldtocsubsubsection=\tocsubsubsection
\renewcommand{\tocsubsubsection}[2]{\hspace{2em}\oldtocsubsubsection{#1}{#2}}

\begin{document}

\title{Ergodicity of non-autonomous discrete systems with non-uniform expansion}

\author[Barrientos]{Pablo G. Barrientos}
\address{\footnotesize \centerline{Instituto de Matem\'atica e Estat\'istica,
Universidade Federal Fluminense, Niter\'oi Brasil}}
\email{barrientos@id.uff.br}

\author[Fakhari]{Abbas Fakhari}
\address{\footnotesize \centerline{Department of Mathematics, Shahid
Beheshti University, Tehran, Iran}}
\email{a\_fakhari@sbu.ac.ir}


\begin{abstract}
We study the ergodicity of non-autonomous discrete dynamical
systems with non-uniform expansion. As an application we get that
any uniformly expanding finitely generated semigroup action of
$C^{1+\alpha}$ local diffeomorphisms of a compact manifold is
ergodic with respect to the Lebesgue measure. Moreover, we will
also prove that every exact non-uniform expandable finitely
generated semigroup action of conformal $C^{1+\alpha}$ local
diffeomorphisms of a compact manifold is Lebesgue ergodic.
\end{abstract}

\maketitle
\thispagestyle{empty}
\section{Ergodicity of finitely generated semigroup actions
with non-uniform expansion}

A local $C^r$-diffeomorphism $f: M  \to M$ of a boundaryless
compact differentiable manifold $M$ is said to be \emph{uniformly
expanding} if in some smooth metric $f$ stretches every tangent
vector. To be precise, if for some choice of a Riemannian metric
$\| \cdot \|$, there is $0<\sigma<1$ such that
$$
\|Df(x)^{-1}\|<\sigma \quad  \text{for all $x\in M$.}
$$
In \cite{SS85}, Sullivan and Shub proved that every $C^{1+\alpha}$
uniformly expanding circle local diffeomorphism is ergodic with
respect to Lebesgue measure. On the other hand, the regularity of
this result cannot be improved. Indeed, Quas constructed in
\cite{Q96} a $C^1$ uniformly expanding map of the circle which
preserves Lebesgue measure, but for which Lebesgue measure is
non-ergodic. Although rather folklore is the extension to greater
dimension of the Sullivan procedure, a rigorous proof that every
$C^{1+\alpha}$ uniformly expanding local diffeomorphisms of $M$ is
Lebesgue-ergodic can be easily deduced from~\cite[Theorem~1.1(c),
pg.~167]{M87}.

We will extend the usual definition of a uniformly expanding map
to a semigroup $\Gamma$ finitely generated by local
diffeomorphisms $f_1,\dots,f_d$. Consider
$\Omega=\{1,\dots,d\}^\mathbb{N}$. For a given sequence
$\omega=\omega_1\omega_2\dots\in \Omega$ we define the orbital
branch corresponding to $\omega$ by
$$
     f^n_\omega =
   f^{}_{\omega_n} \circ \dots \circ f^{}_{\omega_2}\circ f^{}_{\omega_1}  \quad \text{for all $n\geq 1$}.
$$
We say that the \emph{action} of $\Gamma$ on $M$ is \emph{uniformly expanding (along an  orbital branch)}
if there exist $\omega\in \Omega$, $\lambda>1$ and $C>0$ such that for
every $x\in M$,
\begin{equation}
\label{eq:uniform-expanding}
\|Df_\omega^n(x)v\|\geq C \lambda^n \|v\| \quad \text{for
all $v\in T_xM$ and $n\geq 1$.}
\end{equation}
Finitely generated semigroup actions by uniformly expanding maps
have been previously considered in~\cite{RV16}. Observe
that~\eqref{eq:uniform-expanding} is more general an include
semigroup non-necessarily generated by expanding maps.  In order
to extend the above result about the ergodicity of the Lebesgue
measure for random uniformly expanding semigroup actions we need
first some definitions.

A set $A \subset M$ is \emph{$\Gamma$-invariant set} if $f(A)
\subset A$ for all $f\in\Gamma$. We say that the semigroup action
of $\Gamma$ on $M$ is \emph{ergodic with respect to Lebesgue
measure}
if $m(A)\in \{0,1\}$ for all $\Gamma$-invariant set $A$ of $M$ where $m$ denotes the
normalized Lebesgue measure of $M$.

\begin{maintheorem}
\label{thmA} Every uniformly expanding finitely generated
semigroup action of $C^{1+\alpha}$ local diffeomorphisms of a
compact manifold is ergodic with respect to Lebesgue measure.
\end{maintheorem}

The  $C^{1+\alpha}$-regularity assumption behind the ergodicity
theorems essentially related to the {\it bounded distortion
property} which guarantees the preservation of density by the
dynamics. There are many examples that show that $C^1$-regularity
condition alone is not enough (see for instance \cite{B1,Q96}).
For uniformly expanding actions of $C^2$ endomorphisms the Theorem
\ref{thmA} can be deduced from~\cite[Theorem 2.2]{K08}. We will
get this theorem (for $C^{1+\alpha}$ local diffeomorphisms), as a
consequence of the following result which requires to introduce a
generalization of uniformly expanding actions.

We say that the action of $\Gamma$ is \emph{non-uniformly expanding (along an
orbital branch)} if there  is $\omega=\omega_1\omega_2\dots \in
\Omega$ such that for $m$-almost every $x\in M$,
\begin{equation}
\label{eq0}
  \limsup_{n\to\infty} \frac{1}{n} \sum_{i=0}^{n-1}
  \log \|Df^{}_{\omega_{i+1}}(f^i_\omega(x))^{-1}\|<0.
\end{equation}
The action of $\Gamma$ is said to be \emph{exact} if for every
open set $B$ of $M$
there are maps a sequence of maps $(g_n)_n$ in $\Gamma$
such that
$$
     M = \bigcup_{n\in\mathbb{N}} g_n(B) \qquad \text{modulo a set of zero $m$-measure.}  
$$


\begin{maintheorem} \label{thmB} Every exact
non-uniformly expanding finitely generated semigroup action of
  $C^{1+\alpha}$ local diffeomorphisms of a compact manifold is
ergodic with respect to Lebesgue measure.
\end{maintheorem}

It is clear that there are no uniformly expanding semigroup
actions of diffeomorphisms. Indeed, by definition, there exist
$\omega$ and $n\geq 1$ large enough such that
$\|Df^n_\omega(x)^{-1}\|<1$ for all $x\in M$. In other words,
there exists an uniformly expanding map $g$ in the semigroup
$\Gamma$ which forbids $\Gamma$ to be a semigroup of
diffeomorphisms. In fact, we will show that there are no
non-uniformly expanding finitely generated semigroup actions of
diffeomorphisms. Because of this, in \cite{EGZ3} the authors
introduced a weak form of non-uniform expansion.  Namely, they ask
the existence of a constant $a>0$ such that for $m$-almost every
$x\in M$ there is $\omega\in \Omega$ such that
\begin{equation}
\label{eq1}
   \limsup_{n\to\infty} \frac{1}{n} \sum_{i=0}^{n-1}
  \log \|Df^{}_{\omega_{i+1}}(f^i_\omega(x))^{-1}\|<-a.
\end{equation}
In this case we say that the action of $\Gamma$ is
\emph{non-uniformly strong expandable}. They constructed a large
class of examples of semigroup action of diffeomorphisms
satisfying this non-uniform expansion. They proved the ergodicity 
of a finitely generated non-uniformly expanding action with a 
finite Markov partition. However, the existence of finite Markov partitions for 
finitely generated expanding actions seems to be crucial assumption, because they have only finitely 
generated Markov partitions under even strong condition of conformality (see \cite{NGF}).

In the recent paper \cite[Theorem A]{RZ16} Rashid and Zamani claim
that every non-uniformly strong expandable transitive finitely
generated \emph{semigroup} action of conformal \emph{local}
$C^{1+\alpha}$-diffeomorphisms is ergodic with respect to
Lebesgue. However, the proof only works for group actions of
diffeomorphisms (arguments in pg.~8, lines~3-4 in the proof of
Theorem~A cannot be correctly applied for forward invariant sets).
Nevertheless, modifying slightly the assumptions replacing
transitivity by exactness one can recover easily the result for
semigroups. In fact, we will obtain this result assuming a weaker
notion of non-uniformly expansion. Namely, we assume that the
action of $\Gamma$ is \emph{non-uniformly expandable}, that is,
for $m$-almost every $x\in M$ there exists $\omega\in \Omega$ such
that~\eqref{eq0} holds.

\begin{maintheorem} \label{thmC}
Every exact non-uniformly expandable finitely generated semigroup
action of conformal  $C^{1+\alpha}$ local diffeomorphisms of a
compact manifold is ergodic with respect to Lebesgue measure.
\end{maintheorem}
Recall that a local diffeomorphism $g$ is said to be
\emph{conformal} if there exists a function $a:M\to \mathbb{R}$
such that for all $x \in M$ we have that $Dg(x)= a(x) \,
\mathrm{Isom}(x)$, where $\mathrm{Isom}(x)$ denotes an isometry of
$T_xM$. From the above result one obtains as a corollary the main
result of~\cite{BFMS} about the ergodicity of the expanding
minimal semigroup actions of diffeomorphisms. A semigroup action
generated by $C^1$-diffeomorphisms $f_1,\dots,f_d$ of $M$ is said
to be \emph{expanding} if for every $x\in M$ there exists $h$ in
the inverse semigroup (the semigroup generated by inverse maps
$f_1^{-1},\dots,f_d^{-1}$) such that $\|Dh(x)^{-1}\|<1$. It is not
difficult to see that if the semigroup action is expanding and
minimal then action of the inverse semigroup is non-uniformly
expandable and exact. Hence, by the above result one gets that the
action is ergodic with respect to Lebesgue measure whether
$f_1,\dots,f_d$ are conformal $C^{1+\alpha}$-diffeomorphisms
(\cite[Thm.~B]{BFMS}). We provide more details and new examples
where Theorem~\ref{thmC} applies in the last section of this work.

Observe that conditions~\eqref{eq:uniform-expanding},~\eqref{eq0}
and \eqref{eq1} only require the existence of a sequence of
functions satisfying the corresponding property. This is in fact
because the above results are actually a consequence of an
abstract theory in the context of non-autonomous discrete
dynamical systems in compact metric spaces with non-uniform
expansion. In the next section, \S\ref{sec:non-autonomous}, we
will develop this theory and in \S\ref{sec:main-result} we will
provide the main results for non-autonomous discrete dynamical
systems. After that in \S\ref{sec:main-thm-semigroups}, we obtain
as a consequence the above results.

\section{Non-autonomous discrete dynamical systems with non-uniform expansion}
\label{sec:non-autonomous}

A \emph{non-autonomous discrete dynamical system} is a pair $(M,
f_{1,\infty})$ where $M$ is a compact metric space and
$f_{1,\infty}=(f_n)_{n\in\mathbb{N}}$ is a sequence of continuous
maps from $M$ to itself. As it is usual, for each $k \in
\mathbb{N}$, we denote by $f_{k,\infty}$ the sequence of maps
$f_{k+n}:M\to M$ for $n\in \mathbb{N}$ and
$$
   f^{0}_k\eqdef \mathrm{id}  \ \ \text{and} \ \ f^n_k\eqdef
   f^{}_{k+n-1}\circ \dots \circ f^{}_{k+1}\circ
   f^{}_{k}   \qquad n \in \mathbb{N}.
$$
Associated with this system we have a skew-product map $F$ on
$\mathcal{M}=\mathbb{N}\times M$ given by $F(k,x)=(k+1,f_k(x))$.
Observe that $F^n(k,x)=(k+n,f^n_k(x))$ for all $n\geq 0$.

We consider a Borel probability
measure $m$ on $M$  which  is \emph{non-singular} for
$f_{1,\infty}$, that is, both $m(f_n(A)) = 0$ and $m(f^{-1}_n(A))
= 0$ whenever $m(A) = 0$ for all $n \in \mathbb{N}$.
We want to understand the long-term behavior of the fiberwise
orbits
of typical points in $M$ with respect to the measure~$m$.
To do this, we will study \emph{forward $f_{1,\infty}$-invariant}
sets, i.e, measurable sets $A$ so that $f_n(A)\subset A$ for all
$n\in\mathbb{N}$.
Namely, we will study the following definition:


\begin{definition}
\label{local-ergodicity} We say that a measure $m$ is
\emph{locally $f_{1,\infty}$-ergodic}  if for every forward
$f_{1,\infty}$-invariant measurable set $A$ of $M$ with positive
$m$-measure, there exists an open set $B$ of $M$ such that
$m(B\setminus A)=0$. If the measure of $B$ is uniformly bounded
away from zero, we say that $m$ is \emph{locally strong
$f_{1,\infty}$-ergodic}.
This means that there is $\varepsilon>0$ such that for every
forward $f_{1,\infty}$-invariant set $A$ of $M$ with positive
measure there exists an open set $B$ with $m(B)>\varepsilon$ such
that $m(B\setminus A )=0$.
\end{definition}

Firstly we give some basic properties of $m$ which will be useful later.

\begin{lemma}
\label{lemma-lower_comparable} The support of $m$ is a forward
$f_{1,\infty}$-invariant set. Moreover, for any $r>0$ there exists
$b_1(r)>0$ such that $m(B(x,r))> b_1(r)$, for every $x \in
\mathrm{supp}(m)$.
\end{lemma}

\begin{proof} Given $x \in \mathrm{supp}(m)$. At he first, we claim that $f_n(x)$ also belongs to the support of $m$.
By contradiction, assume that every small neighborhood of
$f_n(x)$ has null $m$-measure. Since $m$ is non-singular and $f_n$
is a continuous map this implies that small neighborhoods
of $x$ have also null $m$-measure. This contradicts $x \in \mathrm{supp}(m)$. The second claim is
straightforward. Assume again, by contradiction, that there exists $r>0$
and a sequence $(x_n)_{n\in\mathbb{N}}$  in $\mathrm{supp}(m)$
such that $m(B(x_n,r))\to 0$ as $n\to \infty$. Since
$\mathrm{supp}(m)$ is a compact set, the sequence must accumulate
at some point $z$ in the support of $m$. Then $ m(B(z,r))\leq
\liminf_{n\to \infty} m(B(x_n,r))=0 $ which contradicts $z \in
\mathrm{supp}(m)$.
\end{proof}
In the sequel, we want to study the local ergodicity of
non-autonomous systems. We will review the theory of \emph{hyperbolic preballs}
and \emph{hyperbolic times} introduced by Alves~\cite{A00} for
autonomous systems and extended by Alves and Vilarinho~\cite{AV13}
for random maps under assumptions of non-uniform expansions. This
theory has been deeply studied and generalized in many works
as~\cite{ABV00,ALP05,VV10}. We state it in the context of
non-autonomous systems.
\subsection{Hyperbolic preballs:}
Here, we give two sufficient conditions to get local
ergodicity. This starts by introducing the notion of hyperbolic
pre-balls.
\begin{definition}
Let $\delta>0$ and $0<\lambda<1$. Given $n\geq 1$  and $(k,x)\in
\mathcal{M}$, we say that a neighborhood $V_k^n(x)$ of $x$ in $M$
is a $(\delta,\lambda)$-\emph{hyperbolic preball} of order $n$ of
$f_{1,\infty}$ for the point $(k,x)$ if
\begin{enumerate}
\item \label{con1}  the map $f^{n}_k : M \to M$ sends $V^n_k(x)$ homeomorphically onto the open ball
$B(f_k^n(x),\delta)$ centered at the point $f_k^n(x)$ and of
radius $\delta$,
\item  \label{cond2} for every $y,z\in V_k^{n}(x)$
\begin{equation}
\label{item2}
d(f^i_k(y), f^i_k(z)) \leq \lambda^{n-i} \, d(f^n_k(y), f^n_k(z)) \qquad \text{for $i=0,\dots,n-1$.}
\end{equation}
\end{enumerate}
\end{definition}

\begin{remark} \label{rem-preballs}
Notice that~\eqref{con1} and \eqref{cond2} can be extended to the
closure of $V^n_k(x)$.
\end{remark}

In addition, we will need that the hyperbolic preballs have a good
control of the distortion with respect to the measure $m$. To be
more clear, we give the following definition.
\begin{definition}
Let $\delta>0$ and $0<\lambda<1$. We say that a point $(k,x)\in
\mathcal{M}$ has a \emph{infinitely many
$(\delta,\lambda)$-hyperbolic preballs with bounded distortion} if
there exist a sequence of $(\delta,\lambda)$-hyperbolic preballs
$V^{n_i}_k(x)$  of order $n_i$ where $n_i \to \infty$ and a
constant $K=K(\delta,\lambda,k)>0$ such that for each
$i\in\mathbb{N}$,
\begin{equation}
\label{bounded-distortion-inequality}
     \frac{m(f^{n_i}_k(A))}{m(f^{n_i}_k(B))}
      \leq K  \, \frac{m(A)}{m(B)}
      \qquad \text{for all pair of measurable sets $A,B\subset V^{n_i}_k(x)$.}
\end{equation}
\end{definition}
In what follows, we show local ergodicity under the assumption
that almost every point has infinitely many hyperbolic preballs
with bounded distortion. This assumption can be interpreted in two
different ways. The first criterium will be used to get local
ergodicity of non-uniform expanding non-autonomous systems. The
second will be applied latter for non-uniform expandable
non-autonomous systems.
\subsubsection{First criterium: preballs with bounded distortion}
We assume the existence of a state $\{k\}\times M$ in which almost
every point has infinitely many hyperbolic preballs with bounded
distortion.
\begin{proposition}
\label{prop-ergodic} If there is $k\in \mathbb{N}$ such that for
$m$-almost every point $x\in M$ there exist $\delta=\delta(x)>0$
and $0<\lambda=\lambda(x)<1$ so that $(k,x)$ has infinitely many
$(\delta,\lambda$)-hyperbolic preballs with bounded distortion
then $m$ is locally $f_{1,\infty}$-ergodic.
\end{proposition}
\begin{proof}
Given $\delta>0$ and $0<\lambda<1$, we define
$$
  Z_{\delta,\lambda}= \{(k,x) \in \mathcal{M}:
  \text{the point $(k,x)$ has infinitely many
  $(\delta,\lambda$)-hyperbolic preballs}   \}.
$$
Let $Z$ be the union of $Z_{\delta,\lambda}$ for $0<\delta$ and
$0<\lambda<1$. We denote by $Z(k)$ the section of $Z$ on
$\{k\}\times M$. That is, $Z(k) = \{x\in M: (k,x)\in Z\}$. Since
$Z_{\delta,\lambda} \subset Z_{\delta',\lambda'}$ for any
$0<\delta'\leq \delta$ and $0<\lambda\leq \lambda'<1$ we can write
$$
    Z\eqdef \bigcup_{0<\delta} \bigcup_{0<\lambda<1} Z_{\delta,\lambda} =
    \bigcup_{n>1} Z_{1/n, 1-1/n}.
$$
Let $A$ be a forward $f_{1,\infty}$-invariant 
set of $M$ with positive $m$-measure. Since the support of $m$ is
also forward $f_{1,\infty}$-invariant we can assume that $A\subset
\mathrm{supp}(m)$. We need to show that there exists an open set
$B$ of $M$ so that $m(B\setminus A)=0$. Notice that, by
assumption, there exists $k\in \mathbb{N}$ such that $m(A\cap
Z(k))= m(A)>0$. Thus, there exist $\delta=\delta(A)>0$ and
$0<\lambda=\lambda(A)<1$ such that $\tilde{A}=A \cap
Z_{\delta,\lambda}(k)$ has positive $m$-measure, where $
  Z_{\delta,\lambda}(k)= \{x\in M: (k,x)\in Z_{\delta,\lambda}\}
$. Moreover, since  $\tilde{A}\subset A$ and $A$ is by assumption
forward $f_{1,\infty}$-invariant then $f_{n}(\tilde{A})\subset A$
for all $n\in\mathbb{N}$. Additionally, every point $(k,x)$ where
$x\in \tilde{A}$ has infinitely many $(\delta,\lambda)$-hyperbolic
preballs. The rest of the proof follows the argument of
\cite[Prop.~2.13]{AV13} which is inspired by \cite{ABV00} .

Let $\gamma>0$ be some small number. By the regularity of $m$ and
since $m(\tilde{A})>0$, there is a compact set $\tilde{A}_c
\subset \tilde{A}$ and an open set $\tilde{A}_o \supset \tilde{A}$
such that $m(\tilde{A}_o\setminus \tilde{A}_c) \leq \gamma
m(\tilde{A})$. Notice that, for any $x\in \tilde{A}_c$ we have a
$(\delta,\lambda)$-hyperbolic preball $V_k^n(x)$ of order $n=n(x)$
contained in $\tilde{A}_o$.
Let $W_k^n(x)$ be the part of $V^n_k(x)$ which is sent
homeomorphically by $f^n_k$ onto the open ball
$B(f^n_k(x),\delta/4)$. By compactness, there are $x_1,\dots,x_r
\in \tilde{A}_c$ such that
\begin{equation}
\label{eq:cover}
 \tilde{A}_c \subset W_1 \cup \dots \cup W_r \quad
    \text{where $W_i=W^{n_i}_k(x_i)$ and $n_i=n(x_i)$ for $i=1,\dots,r$. }
 \end{equation}
Assume that
$$
      \{n_1, \dots, n_r\} = \{ n_1^*,\dots,n_s^*\} \quad   \text{with $n_1^*<\dots<n_s^*$}.
$$
Let $I_1$ be the maximal subset of $\{1,\dots,r\}$ such that for
each $i\in I_1$ both $n_i=n_1^*$ and $W_i \cap W_j = \emptyset $
for every $j \in I_1$ with $j\not=i$. Inductively we define
$I_\ell$ for $\ell=2,\dots,s$ as follows: supposing that $I_1,
\dots, I_{\ell-1}$ have already been defined, let $I_\ell$ be a
maximal set of $\{1,\dots ,r\}$ such that for each $i \in I_\ell$
both $n_i =n_\ell^*$ and $W_i \cap W_j =\emptyset$ for every $j
\in I_1 \cup \dots \cup I_k$ with $i \not = j$. Set  $I = I_1 \cup
\dots \cup I_s$. By construction, we have that $W_i$ for $i\in I$
are pairwise disjoint sets.

We will prove that the family of set $V_i=V_k^{n_i}(x_i)$ for $i\in I$ covers $\tilde{A}_c$.
Indeed, by construction, given any $W_j$ with $j=1,\dots,r$, there is
some $i \in I$ with $n_i \leq n_j$ such that $W_j \cap W_i \not =\emptyset$.
Taking images by $f^{n_i}_k$, we have $f^{n_i}_k(W_j)\cap B(f^{n_i}_k(x_i),\delta/4)\not=\emptyset$.
Since $W_j$ is contained in the $(\delta,\lambda)$-hyperbolic preball $V_j$ of order $n_j$  and $n_i \leq n_j$, by definition of hyperbolic preballs,
$$
\diam (f_k^{n_i}(W_j))\leq \lambda^{n_j - n_i} \diam (f_k^{n_j}(W_j))  \leq \frac{\delta}{2}. 
$$
Hence $f^{n_i}_k(W_j)\subset B(f^{n_i}_k(x_i),\delta)$.
This gives that $W_j \subset V_i$. Taking into account \eqref{eq:cover}, we get that the family of sets
$V_i$ for $i\in I$ covers $\tilde{A}_c$.

Observe that by the bounded distortion property \eqref{bounded-distortion-inequality} applied to $A=W_i$ and $B=V_i$
we get $K=K(\delta,\lambda,k)>0$ such that
$$
     m(W_i) \geq  K^{-1} \frac{m(B(f^{n_i}_k(x_i),\delta/4))}{m(B(f^{n_i}_k(x_i),\delta))} \, m(V_i).
$$
According to Lemma~\ref{lemma-lower_comparable}, the measure of any ball centered at
a point in the support of $m$ is lower comparable with its radius and thus we can find a constant
$\tau=\tau(\delta,\lambda,k)>0$ so that $m(W_i)\geq \tau m(V_i)$ for all $i\in I$. Hence
\begin{align*}
m(\bigcup_{i\in I} W_i) = \sum_{i\in I} m(W_i) \geq \tau \sum_{i\in I} m(V_i) \geq \tau m(\bigcup_{i\in I} V_i) \geq \tau m(\tilde{A}_c) \geq \frac{\tau}{2} m(\tilde{A}).
\end{align*}
The last inequality is obtained from the fact that $m(\tilde{A}_c)>(1-\gamma) m(\tilde{A})$ and
choosing $\gamma>0$ small enough which it is possible because the constant $\tau$ does not depend on $\gamma$.
Now, we are going to prove the existence of $i\in I$ in such away that
\begin{equation}
\label{eq:gamma}
  \frac{m(W_i\setminus \tilde{A})}{m(W_i)} < \frac{2\gamma}{\tau}.
\end{equation}
Indeed, otherwise we get the following contradiction.
\begin{align*}
\gamma m(\tilde{A}) \geq m(\tilde{A}_o\setminus \tilde{A}_c) \geq m(\bigcup_{i\in I} W_i \setminus \tilde{A})
\geq \frac{2\gamma}{\tau} m(\bigcup_{i\in I} W_i) > \gamma m(\tilde{A}).
\end{align*}
Finally, we obtain the required open ball $B$. Since $f^{n_i}_k(\tilde{A})\subset A$ and $f^{n_i}_k$ is injective on $W_i$, we have
$$
   m(f^{n_i}_k(W_i)\setminus A) \leq m(f^{n_i}_k(W_i)\setminus f^{n_i}_k(\tilde{A}))= m(f^{n_i}_k(W_i\setminus{\tilde{A}})).
$$
By the distortion property, relation \eqref{eq:gamma} and taking in mind that $f^{n_i}_k(W_i)=B(f^{n_i}_k(x),\delta/4)$, we get
 $$
\frac{m(B(f^{n_i}_k(x),\delta/4)\setminus A)}{m(B(f^{n_i}_k(x),\delta/4))} \leq
 \frac{m(f^{n_i}_k(W_i\setminus \tilde{A}))}{m(f^{n_i}_k(W_i))}
 \leq K \frac{m(W_i\setminus \tilde{A})}{m(W_i)} \leq \frac{2 K \gamma}{\tau},
 $$
which can obviously be made arbitrarily small, letting $\gamma\to
0$. From this, one easily deduces, taking an accumulation point of
this balls, that there is a  ball $B$ of radius $\delta/4$ where
the relative measure of $A$ is one. This completes the proof.
\end{proof}
\begin{remark} \label{rem-ergodic}  In the above proof, the radius of the obtained open
ball depends only on $\delta>0$ but it may be vary from an invariant set to another one. To get
strong local $f_{1,\infty}$-ergodicity, we must ask that $\delta>0$ and $0<\lambda<1$,
in the statement of the proposition, are
uniform on $x$. In other words, we need that $m(Z_{\delta,\lambda}(k))=1$, for some $k\in\mathbb{N}$.
\end{remark}


\subsubsection{Second criterium: preballs with regularity}  Now, we assume that almost
every point $x$ has infinitely many hyperbolic preballs, but
probably in different states $\{k\}\times M$. This assumption is
obviously weaker than the previous condition. To prove the local
ergodicity, we also need to assume that the preballs have a good
control of the regularity.

\begin{definition}
Let $\delta>0$ and $0<\lambda<1$. We say that a point $(k,x)\in
\mathcal{M}$ has \emph{infinitely many regular
$(\delta,\lambda)$-hyperbolic preballs} if there exist a sequence
of $(\delta,\lambda)$-hyperbolic preballs $V_i=V^{n_i}_k(x)$  of
order $n_i$ where $n_i \to \infty$ and a constant
$L=L(\delta,\lambda,k)>0$ such that
\begin{equation}
\label{conformal-def}
      m(B(x,R_i))\leq L \, m(B(x,r_i)), \quad \text{for all $i\in\mathbb{N}$}
\end{equation}
where $B(x,R_i)$ and $B(x,r_i)$ are, respectively, the smallest ball around $x$ containing $V_i$ and the largest
ball around $x$ contained in $V_i$.
\end{definition}
The following proposition shows local ergodicity under the
assumption of the existence of infinitely many regular preballs
with bounded distortion. Here, we also need to assume that the
metric measure space $(M,d,m)$
satisfies the 
\emph{density point property}. That is, for any measurable set $A$
of
$M$, 
\begin{equation}\label{eq:density}
 \lim_{r\to 0^+ } \frac{m(A\cap B(x,r))}{m(B(x,r))}=1, \quad
 \text{ for $m$-almost all $x\in A$.}
\end{equation}
This property holds in any metric space for which Besicovitch's Covering Theorem holds. In particular, it holds for
any Borel probability measure in Euclidean spaces. Also, it
is satisfied for any Borel probability measure in a Polish
ultra-metric space and for the Cantor space $2^\mathbb{N}$ with the
coin-tossing measure and the usual distance. In general metric
spaces this is not necessarily the case~\cite{KRS16}. As another relatively general mode of this property, one can
refer to the \emph{weak locally doubling} measure $m$ (see~\cite[Thm.~3.4.3]{HKST15}) in
the sense that
$$
    \limsup_{r\to 0^+} \frac{m(B(x, 2 r))}{m(B(x,r))} <\infty,
    \quad \text{for $m$-almost all $x\in M$.}
$$

\begin{proposition}
\label{prop-ergodic-conformal} Assume that $(M,d,m)$ satisfies the
density point property. If for $m$-almost every point $x\in M$
there are $k=k(x)\in\mathbb{N}$, $\delta=\delta(x)>0$ and
$0<\lambda=\lambda(x)<1$ so that $(k,x)$ has infinitely many
regular $(\delta,\lambda$)-hyperbolic preballs with bounded
distortion then $m$ is locally $f_{1,\infty}$-ergodic.
\end{proposition}
\begin{proof} Let $A$ be a forward $f_{1,\infty}$-invariant set
with positive $m$-measure. By the density point property
$m$-almost every point in $A$ is a density point. That is, it
satisfies~\eqref{eq:density}.
By the assumption we find a density point $x\in A$, $k=k(x)\in
\mathbb{N}$, $\delta=\delta(x)>0$ and $0<\lambda=\lambda(x)<1$ so
that $(k,x)$  has a nested sequence of regular
$(\delta,\lambda)$-preballs $V_i=V^{n_i}_k(x)$ of order $n_i\to
\infty$ with bounded distortion. Let $z$ be an accumulation point
of $f_k^{n_i}(x)$. Then taking a subsequence if it is necessary we
have $B(z,\delta/2)\subset B(f_k^{n_i}(x),\delta)$ for all $i$
large enough.  On the other hand, since $f^{n_i}_k(A)\subset A$
and $f^{n_i}_k$ is injective on $V_i$, we have
$$
   m(f^{n_i}_k(V_i)\setminus A) \leq m(f^{n_i}_k(V_i)\setminus f^{n_i}_k(A))= m(f^{n_i}_k(V_i\setminus{A})).
$$
By the distortion property and the regularity of the preballs,
having into account that $f^{n_i}_k(V_i)=B(f^{n_i}_k(x),\delta)$,
it follows that for every $i$ large enough
 \begin{align*}
\frac{m(B(z,\delta/2)\setminus A)}{m(M)} &\leq
 \frac{m(f^{n_i}_k(V_i)\setminus A)}{m(f^{n_i}_k(V_i))}
 \leq  \frac{m(f^{n_i}_k(V_i \setminus A))}{m(f^{n_i}_k(V_i))}
 \leq K \, \frac{m(V_i\setminus A)}{m(V_i)} \\
 &\leq K \, \frac{m(B(x,R_i)\setminus A)}{m(B(x,R_i))} \cdot
 \frac{m(B(x,R_i))}{m(B(x,r_i))}
 \leq K \, L \,  \frac{m(B(x,R_i)\setminus A)}{m(B(x,R_i))}
 \end{align*}
 where $B(x,R_i)$ and $B(x,r_i)$ are, respectively, the smallest ball around $x$ containing $V_i$ and the largest ball around $x$ contained in $V_i$. Taking limit as $i\to \infty$, since $V_i$ is nested then $R_i\to 0$ and since $x$ is a density point of $A$, we get that $m(B(z,\delta/2)\setminus A)=0$.
 This completes the proof of the proposition.
\end{proof}

\begin{remark} \label{rem-ergodic-conformal}  
To get strong local $f_{1,\infty}$-ergodicity it suffices to ask that $\delta>0$ and $0<\lambda<1$ in the statement of the proposition are uniform on $x$. 
\end{remark}

\begin{remark} \label{rem1} Proof of Proposition~\ref{prop-ergodic-conformal}
actually shows the following: if $x$ is a density point of a
$f_{1,\infty}$-invariant set $A$ such that there is $k=k(x)\in
\mathbb{N}$, $\delta=\delta(x)>0$ and $0<\lambda=\lambda(x)<1$
then there is $z$ such that $m(B(z,\delta/2)\setminus A)=0$.
\end{remark}

\subsection{Hyperbolic preballs with bounded distortion} 
Here, we will study how we can get hyperbolic preballs with
bounded distortion. First we need the following lemma.

\begin{lemma}\label{lem:distor}
For each $n\in\mathbb{N}$, consider functions $\psi_n:M\to
\mathbb{R}$ and assume that there exist $k\in \mathbb{N}$,
$0<\alpha\leq 1$, $\epsilon>0$ and a constant
$C_k=C_k(\epsilon,\alpha)>0$ such that
\begin{equation}\label{eq:locally-Holder}
    |\psi_n(x)-\psi_n(y)| \leq C_k \, d(x,y)^\alpha, \quad
    \text{for all $x,y\in M$ with $d(x,y)<\epsilon$ and $n\geq k$.}
\end{equation}
Then any $(\delta,\lambda)$-preball $V^n_k(x)$ of order $n$ for a
point $(k,x)\in \mathcal{M}$ with $0<\delta \leq \epsilon$,
$0<\lambda<1$ there is a constant $K=\exp(C_k\delta^\alpha
(1-\lambda^\alpha)^{-1})>0$ such that
\begin{equation*}
     K^{-1} \leq e^{S_n\psi(k,y)-S_n\psi(k,z)} \leq K,
     \quad \text{for all $y,z\in \overline{V_k^n(x)}$}
\end{equation*}
where
$$
S_n \psi = \sum^{n - 1}_{i=0} \psi \circ F^i
$$
denotes the $n$-th Birkhoff sum of a function $\psi:
\mathcal{M}\to \mathbb{R}$ given by $\psi(k,x)=\psi_k(x)$.
\end{lemma}
\begin{proof} For any pair of points $y,z\in \overline{V_k^n(x)}$, by
definition of $(\lambda,\delta)$-hyperbolic preball (see also
Remark~\ref{rem-preballs}),
$$
d(f_k^i(y),f_k^i(z)) \leq \lambda^{n-i} d(f^n_k(y),f^n_k(z)) \leq
\lambda^{n-i } \delta \leq \epsilon \quad \text{for all
$i=0,\dots,n-1$}
$$
and thus
\begin{align*}
|S_n\psi(k,y)-S_n\psi(k,z)| &\leq \sum_{i=0}^{n-1} |\psi_{k+i}(f^i_k(y))-\psi_{k+i}(f^i_k(z))| \\
&\leq \sum_{i=0}^{n-1} C_k\, d(f^i_k(y),f^i_k(z))^\alpha \leq
\sum_{i=0}^{n-1} C_k \lambda^{(n-i)\alpha}  \delta^\alpha.
\end{align*}
It is then enough to take $K= \mathrm{exp}(\sum_{j=0}^{\infty} C_k
\lambda^{j \alpha} \delta^\alpha)= \exp(C_k\delta^\alpha
(1-\lambda^\alpha)^{-1})>0$.
\end{proof}

In order to get the bounded distortion property we will need to
suppose that the measure  $m$ is \emph{$f_{1,\infty}$-conformal}. That is,
for each $n\in\mathbb{N}$ we have some function $\psi_n: M \to
\mathbb{R}$ such that
$$
m(f_n(A)) = \int_{A} e^{-\psi_n(x)} \, dm(x), \quad
\text{for every measurable set $A$ so that $f_n|_A$ is injective.}
$$
Surely, any absolutely continuous measure is conformal, by the definition.
Also, there are several examples of conformal measures appearing in the literature  
(see \cite{DU91}, for a large class of examples). 

In fact, the concept of $f_{1,\infty}$-conformal measure allows us to have
varying Jacobians with respect to the dynamics in the sequence.

\begin{proposition} 
\label{prop-distortion} Assume that $m$ is $f_{1,\infty}$-conformal as above
and there exists $k\in \mathbb{N}$, $0<\alpha\leq 1$, $\epsilon>0$
such that the functions $(\psi_n)_n$ satisfy the locally H\"older
condition~\eqref{eq:locally-Holder}.
Then any $(\delta,\lambda)$-preball of a point $(k,x)\in
\mathcal{M}$ with $0<\delta \leq \epsilon$, $0<\lambda<1$ has
bounded distortion, i.e., satisfies
\eqref{bounded-distortion-inequality} with distortion constant
$K=K(\delta,\lambda,C_k)$ uniform on $x$ and on the order of the
preball.
\end{proposition}

\begin{proof}
We consider $0<\delta\leq \epsilon$, $0<\lambda<1$ and a
$(\delta,\lambda)$-hyperbolic preball $V_k^n(x)$ of order $n$ for
a point $(k,x)\in \mathcal{M}$.
Let $A, B$ be a pair of measurable sets in $V_k^n(x)$. By the
conformality of the measure, it is not hard to see that
$$
    m(f_k^n(A))= \int_A e^{-S_n\psi(k,z)} dm(z) \leq  \sup_{z\in A} e^{-S_n\psi(k,z)} \, m(A)
$$
and
$$
    m(f_k^n(B))= \int_B e^{-S_n\psi(k,y)} dm(y) \geq  \inf_{y\in B} e^{-S_n\psi(k,z)} \, m(B)
$$
where $S_n\psi$ denotes the $n$-th Birkhoff sum of a function
$\psi: \mathcal{M}\to \mathbb{R}$ given by $\psi(k,x)=\psi_k(x)$.
From this and Lemma~\ref{lem:distor} one easily concludes the
proposition.
\end{proof}


\subsection{Hyperbolic times} Now, we will provide a sufficient
condition to get a hyperbolic preball. In order to do this, we
first need to restrict the class of non-autonomous discrete
dynamical systems $f_{1,\infty}=(f_n)_n$ that we are considering.

We suppose that $f_n: M\to M$ for all $n\in\mathbb{N}$ are
\emph{local homeomorphisms with uniform Lipschitz constant for the
inverse branches}. This means that there is a function
$\varphi:\mathcal{M}\to \mathbb{R}$ such that for each $(k,x) \in
\mathcal{M}$ there exists a neighborhood $V$ of $x$ so that $f_k:V
\to f_k(V)$ is invertible and
\begin{equation*}
\label{eq:uniform-lipschitz} d(y,z) \leq \varphi(k,x) \,
d(f_k({y}),f_k(z)), \quad \text{for all for every $y,z \in V$}.
\end{equation*}
%

\begin{definition}
Let $0<\sigma<1$. A positive integer $n\in\mathbb{N}$ is called
\emph{$\sigma$-hyperbolic time} of $f_{1,\infty}$ for the point
$(k,x)\in \mathcal{M}$ if
$$
   \prod_{i=n-\ell}^{n-1} \varphi(F^i(k,x)) \leq \sigma^\ell, \quad \text{for
   $\ell=1,\dots,n$} \ \ \text{where} \ \
 F^i(k,x)=(k+i,f_k^i(x)).
$$
\end{definition}

The following proposition shows that existence of hyperbolic times
implies the existence of hyperbolic preballs.

\begin{proposition}
\label{prop-delta} For any $\epsilon>0$ there is $0<\delta_k\leq
\epsilon$ such that if $n\in\mathbb{N}$ is a $\sigma$-hyperbolic
time of $f_{1,\infty}$ for a point $(k,x)\in \mathcal{M}$ then
$(k,x)$ has a $(\delta_k,\lambda)$-hyperbolic preball of order $n$
where $\lambda=\sigma$.
\end{proposition}
\begin{proof}
First of all we will set $\delta_k>0$. To do this we fix
$\epsilon>0$. For each $k\in \mathbb{N}$, since $f_k$ is a local
homeomorphism, for every $x \in M$ there is
$0<\delta_{k,x}\leq\epsilon$ such that $f_k$ sends a neighborhood
$U(k,x)$ of $x$ homeomorphically onto an open ball of radius
$\delta_{k,x}$ centered at $f_k(x)$ and satisfying
\begin{equation}
\label{eq:uniform}
    d(y,z)\leq \varphi(k,x)\, d(f_k(y),f_k(z)) \quad \text{for all $y,z\in U(k,x)$}.
\end{equation}
By compactness of $M$, we can choose a uniform radius
$\delta_k>0$. Otherwise we find a sequence of points $x_n\in M$
converging to a point $\bar{x}$ and with $\delta_{k,x_n}\to 0$.
Hence, we obtain that $\delta_{k,\bar{x}}$ must to be zero
obtaining a contradiction. Thus we get that $f_k:U(k,x)\to
B(f_k(x),\delta_k)$ is a homeomorphism
satisfying~\eqref{eq:uniform}. Moreover, without loss of
generality, using the order of $\mathbb{N}$, we can assume that
$\delta_k\geq \delta_{k+1}$ for all $k\in \mathbb{N}$.

Now we will show the proposition by induction on $n$. Let $n=1$ be
a $\sigma$-hyperbolic time of a point $(k,x)$. This implies that
$\varphi(k,x)\leq \sigma$. Let $V^1_k(x)$ be the neighborhood
$U(k,x)$ of $x$ obtained  above. Hence we have that $f_k$ sends
homeomorphically $V_k^1(x)$ onto the open ball
$B(f_k(x),\delta_k)$ and
$$
    d(y,z)\leq \varphi(k,x) \, d(f_k(y),f_k(z)) \leq \sigma \, d(f_k(y),f_k(z)), \quad
    \text{for all $y,z\in V^1_k(x)$.}
$$
Thus, $V^1_k(x)$ is a $(\delta_k,\sigma)$-hyperbolic preball of
order $n=1$ at the point $(k,x)$.

Now, assuming the proposition holds for $n$, we prove it for
$n+1$. Namely, we assume that if $n$ is a $\sigma$-hyperbolic time
of a point $(k,x)$, there exists a $(\delta_k,\sigma)$-hyperbolic
preball $V^{n}_k(x)$ and additionally it holds that
$$
    f^{i}_k(V^{n}_k(x)) \subset U(F^i(k,x)),
    \quad \text{for $i=0,\dots,n-1$}.
$$

Let $n+1$ be a $\sigma$-hyperbolic time of a point $(k,x)$. Hence,
$$
     \prod_{i=n-\ell}^{n-1} \varphi(F^{i}(F(k,x))) = \prod_{j=n+1-\ell}^{n} \varphi(F^{j}(k,x)) \leq \sigma^\ell,
     \quad \text{for $\ell=1,\dots,n$}
$$
and thus $n$ is a $\sigma$-hyperbolic time of the point
$F(k,x)=(k+1,f_k(x))$. By induction, there exists a
$(\delta_{k+1},\sigma)$-hyperbolic preball $V$ of order $n$ at the
point $F(k,x)$. This means that $f^n_{k+1}$ sends homeomorphically
$V$ onto $B(f_{k+1}^{n}(f_k(x)),\delta_{k+1})$ and
\begin{equation}
\label{hip:induction}
   d(f^i_{k+1}(\bar y), f^i_{k+1}(\bar z)) \leq \sigma^{n-i}
   d(f^n_{k+1}(\bar y), f^n_{k+1}(\bar z)), \quad \text{for all $\bar y, \bar z \in V$, $i=0,\dots,n-1$.}
\end{equation}
Notice that, in fact, $V\subset B(f_k(x),\delta_k)$ since applying
the above inequality for $i=0$ and recalling that $\delta_k\geq
\delta_{k+1}$, we have that
$$
   d(\bar y, f_k(x)) \leq \sigma^{n}
   d(f^n_{k+1}(\bar y), f^{n}_{k+1}(f_k(x)))
   \leq \sigma^{n}\delta_{k+1}<\delta_k.
$$
Therefore, there is a neighborhood $V^{n+1}_k(x)$ of $x$ which is
sent homeomorphically by $f_k$ onto~$V$. Moreover,
$V^{n+1}_k(x)\subset U(k,x)$. On the other hand, by the induction
hypothesis, we have also that
$$
    f^{i}_k(V^{n+1}_k(x))=f^{i-1}_{k+1}(V) \subset U(F^{i-1}(F(k,x)))=U(F^i(k,x))  \quad \text{for $i=1,\dots,n$}.
$$
Now, we must show that for every $y,z\in V^{n+1}_k(x)$ it holds
that
\begin{equation}
\label{eq:check}
  d(f^j_{k}(y), f^j_{k}(z)) \leq \sigma^{n+1-j}
   d(f^{n+1}_{k}(y), f^{n+1}_k(z)), \quad \text{for all $j=0,\dots,n$.}
\end{equation}
Applying~\eqref{hip:induction} we obtain~\eqref{eq:check} for
$j=1,\dots,n$. Thus, it is enough to check it for $j=0$.
This follows applying recursively
$$
  d(y,z)\leq  \varphi(k,x) d(f_{k}(y), f_{k}(z))
  \leq \dots \leq \prod_{i=0}^{n}  \varphi(F^i(k,x))
  d(f^{n+1}_{k}(y), f^{n+1}_{k}(z))
$$
for any $y,z\in V^{n+1}_k(x)$. Since $n+1$ is a
$\sigma$-hyperbolic time of $(k,x)$ we complete the proof.
\end{proof}
\subsection{Expanding/expandable measures}  In this subsection, we study how to get hyperbolic times.
We will continue assuming that $f_{1,\infty}=(f_n)_n$ is a
non-autonomous discrete system of local homeomorphisms $f_n$ with
uniform Lipschitz constant $\varphi(n,x)=\varphi_n(x)$ for the
inverse branches as in the previous section. Additionally we
assume that
$$
      \sup\{-\log \varphi(k,x): x\in M, k\in \mathbb{N}\} < \infty.
$$
For $k\in\mathbb{N}$ and $a>0$, let $M(k,a)$ be the set of
points $x\in M$ such that
$$
    \limsup_{n\to \infty} \frac{1}{n}\sum_{i=0}^{n-1} \log \varphi (F^i(k,x)) <-a.
$$
\begin{proposition} \label{prop-infinite-hyperbolic-times} If $x\in M(k,a)$ then there is $\sigma=\exp(-a/2)$ such that  $(k,x)$ has
infinitely many $\sigma$-hyperbolic times.
\end{proposition}
\begin{proof}
For every $x \in M(k,a)$ and $N$ sufficiently large
we have
$$
    \sum_{i=0}^{N-1} -\log \varphi (F^i(k,x)) \geq Na.
$$
Taking $a_i = -\log \varphi(F^i(k,x))- a/2$,  we have  $a_0 + \dots + a_{N-1} \geq a N /2$.
By Pliss  lemma (c.f.~\cite[Lemma~4.2]{AV13}) with
$$
   c = a/2 \quad \text{and} \quad A = \sup\{-\log \varphi(i,z)-a/2 : z \in M, i\in \mathbb{N}\}
   <\infty,
$$
there are $t \geq \theta N$, $\theta = c/A$ and  $1 \leq n_1 < \dots < n_\ell \leq N$ such that
$$
     \sum_{i=n}^{n_j-1} a_i \geq 0,  \quad \text{for $n=0,\dots,n_j-1$ and $j=1,\dots,t$}.
$$
Therefore,
$$
 \sum_{i=n}^{n_j-1} \log \varphi(F^i(k,x)) \leq \frac{a}{2} (n_j-n), \quad \text{for $n=0,\dots,n_j-1$ and $j=1,\dots,t$}.
$$
By taking $0<\sigma=\exp(-a/2)<1$ and $\ell=n_j-n$, we get
$$
   \prod_{i=n_j-\ell}^{n_j-1} \varphi(F^i(k,x)) \leq \sigma^\ell \quad \text{for $\ell=0,\dots,n_j-1$ and $j=1,\dots,t$}.
$$
This implies that $n_j$ for $j=1,\dots,t$ are $\sigma$-hyperbolic times of $F$ for $(k,x)$. Since $t\to \infty$
as $N\to \infty$ we obtain infinitely many hyperbolic times and complete the proof.
\end{proof}
Following \cite{V11}, we say that a measure $m$ is
\emph{$f_{1,\infty}$-expanding} if there is $k\in \mathbb{N}$ so
that
$$
    \limsup_{n\to \infty} \frac{1}{n}\sum_{i=0}^{n-1} \log \varphi (F^i(k,x)) <0,
    \quad \text{for $m$-almost every $x\in M$}.
$$
Observe that equivalently, one can ask that the above limit holds
at $(1,x)$, for $m$-almost every $x\in M$. If the limit is
uniformly far away from zero, as in the above proposition, i.e.,
if there is $a>0$ such that $m(M(1,a))=1$, we say that $m$ is
\emph{strong $f_{1,\infty}$-expanding}.

Similarly,  we will say that $m$ is
\emph{$f_{1,\infty}$-expandable} if for $m$-almost every $x\in M$
there is $k=k(x)\in \mathbb{N}$ such that
$$
    \limsup_{n\to \infty} \frac{1}{n}\sum_{i=0}^{n-1} \log \varphi (F^i(k,x)) <0.
$$
In addition, if there is $a>0$ uniform on $x$ so that for
$m$-almost every $x\in M$ there is $k=k(x)\in \mathbb{N}$ such
that $x \in M(k,a)$, then we say that $m$ is \emph{strong
$f_{1,\infty}$-expandable}.

As a consequence of the above proposition, we have the following:

\begin{corollary}     It holds that,
\label{cor-expanding}
\begin{enumerate}
\item if $m$ is $f_{1,\infty}$-expanding (resp.~strong $f_{1,\infty}$-expanding) then for $m$-almost every $x\in M$
there exists $0<\sigma=\sigma(x)<1$ (resp.~$0<\sigma<1$ uniform on
$x$) such that the point $(1,x)\in \mathcal{M}$ has infinitely
many $\sigma$-hyperbolic times.
\item if $m$ is $f_{1,\infty}$-expandable (resp.~strong $f_{1,\infty}$-expandable) then for $m$-almost every $x\in M$
there exist $k=k(x)\in\mathbb{N}$ and $0<\sigma=\sigma(x)<1$
(resp.~$0<\sigma<1$ uniform on $x$) such that $(k,x)$ has
infinitely many $\sigma$-hyperbolic times.
\end{enumerate}
\end{corollary}

\subsection{Locally geodesic metric spaces}
A metric space is said to be \emph{locally geodesic} (or locally
$1$-quasiconvex) if each point has a neighborhood $U$ such that
for each pair of points $x, y \in U$, there is a rectifiable curve
$\gamma$ joining $x$ and $y$  with length $\ell(\gamma)=d(x,y)$.
In this subsection, we assume that $(M,d)$ is locally geodesic and
we show how expanding/expandable measures and regular hyperbolic
preballs can be obtained in this case.

\subsubsection{Expanding/expandable measures}
\label{sec:local-geodesic-expanding} In the two previous
subsection, we assumed that the non-autonomous system
$f_{1,\infty}=(f_n)_n$, formed by local homeomorphisms $f_n$ have
\emph{uniform} Lipschitz constant $\varphi(n,x)=\varphi_n(x)$ for
the inverse branches. That is,
satisfying~\eqref{eq:uniform-lipschitz}. An a priori weaker
condition is to assume that maps $f_n$ have \emph{pointwise}
Lipschitz constants $\theta(n,x)=\theta_n(x)>0$ for the inverse
branches. That is, there is a positive bounded functions
$\theta_n: M\to \mathbb{R}$ such that for each $x \in M$ it holds
$$
  \theta_n(x)^{-1}=\liminf_{y\to x}
  \frac{d(f_n(x),f_n(y))}{d(x,y)}.
$$
According to~\cite[Cor.~2.4]{DJ10}, any pointwise Lipschitz map on
a locally geodesic metric space $(M,d)$ is uniformly Lipschitz.
Moreover, by \cite[Lemma~2.3]{DJ10}, restricting $f_n$ to a small
neighborhood $V$ of $x$, one gets
\begin{align*}
d(y,z) \leq \|\theta_n\|^{}_{\infty,V} \, d(f_n(y),f_n(z)), \quad
\text{for all $y,z\in V$.}
\end{align*}
Thus we can take $\varphi(n,x)= \|\theta_n\|^{}_{\infty,V}$. In
addiction, if $\theta:\mathcal{M}\to \mathbb{R}$,
given by $\theta(n,x)=\theta_n(x)$ is a continuous function (with
the discrete topology in $\mathbb{N}$) or equivalently,
$\theta_n: M \to \mathbb{R}$ is a continuous map, for all $n\in
\mathbb{N}$ then one can get also an upper estimative. Indeed, since
$\theta_n(x)<\sigma^{-1/2}\theta_n(x)$, for all $0<\sigma<1$, by
the continuity, one can find a small neighborhood $V=V(\sigma)$ of
$x$ such that
\begin{equation*}
    \varphi(n,x)=\|\theta_n\|^{}_{\infty,V}
    \leq \sigma^{-1/2}\theta_n(x)=\sigma^{-1/2}\theta(n,x).
\end{equation*}
Hence,
\begin{align*} \label{eq:borra}
    \limsup_{n\to \infty} \frac{1}{n}\sum_{i=0}^{n-1}  \log
    \theta(F^i(k,x))
     &\leq \limsup_{n\to \infty} \frac{1}{n}\sum_{i=0}^{n-1} \log
     \varphi(F^i(k,x))
    \\ &\leq \limsup_{n\to \infty} \frac{1}{n}\sum_{i=0}^{n-1} \log
    \theta(F^i(k,x)) - \frac{1}{2}\log\sigma.
\end{align*}
Consequently, by taking $\sigma=\sigma(x)$ close enough to one, we
get the following:
\begin{proposition} \label{prop-M}
Given $x\in M$, there is $a=a(x)>0$ such that $x\in M(k,a)$ if and
only if
\begin{equation*} 
\limsup_{n\to \infty} \frac{1}{n}\sum_{i=0}^{n-1}  \log
    \theta(F^i(k,x))<0.
\end{equation*}
Moreover, $m$ is (strong) $F$-expanding/expandable if and only if
it holds
\begin{equation*}
\limsup_{n\to \infty}
\frac{1}{n}\sum_{i=0}^{n-1} \log
    \theta(F^i(k,x))< -a
\end{equation*}
under the corresponding quantification assumptions and $a\geq 0$.
\end{proposition}

\subsubsection{Regular hyperbolic preballs } Next, we are
going to show how we can get regular preballs. To do this, we need
to imose some extra conditions on the metric measure space
$(M,d,m)$ and also on the non-autonomous dynamical systems
$f_{1,\infty}=(f_n)_{n\in\mathbb{N}}$.

We will assume that the measure $m$ is \emph{locally doubling}
(see~\cite[pg.~326]{HKST15}), i.e., there are $\rho>0$ and $L>0$
such that
$$
m(B(x,2r)) \leq L \, m(B(x,r))
$$
for each $x\in M$ and each $0<r \leq \rho$.
Every locally doubling metric measure space
satisfies the density point property.

Finally, we will impose that $f_{1,\infty}=(f_n)_{n\in\mathbb{N}}$
is \emph{conformal} in the sense that $f_n$ is a conformal map for
all $n\in\mathbb{N}$. Namely, there is a function $\phi_n: M\to
\mathbb{R}$ such that for every $x\in M$
\begin{equation*}
\lim_{y\to x} \frac{d(f_n(x),f_n(y))}{d(x,y)} = e^{-\phi^{}_n(x)}.
\end{equation*}
Observe that, in this case
$$
      \theta_n(x)=e^{\phi^{}_n(x)}, \quad \text{for all $x\in M$ and $n\in\mathbb{N}$}.
$$


\begin{proposition}
\label{prop-conformal} Let $f_{1,\infty}=(f_n)_{n\in\mathbb{N}}$
be a conformal non-autonomous discrete system  on a locally
geodesic compact metric measure space $(M,d,m)$, where $m$ is
locally doubling. If there are $k\in \mathbb{N}$, $0<\epsilon$,
$0<\alpha\leq 1$ and $C_k=C_k(\epsilon,\alpha)>0$ such that
$$
    |\phi^{}_n(x)-\phi^{}_n(y)| \leq C_k \, d(x,y)^\alpha, \quad
    \text{for all $x,y\in M$ with $d(x,y)<\epsilon$ and $n\geq k$,}
$$
then any $(\delta,\lambda)$-hyerbolic preball of order $n$ of a
point $(k,x)\in \mathcal{M}$ with $0<\delta \leq \epsilon$,
$0<\lambda<1$ is regular, i.e., satisfies~\eqref{conformal-def}
with regularity constant $L(\epsilon,\lambda,k)>0$, uniform on $x$
and on the order of the pre-ball.
\end{proposition}

\begin{proof}
At the first, note that by compactness of $M$, one can assume that any ball
of radius less than  $\epsilon>0$ is contained in a geodesic
neighborhood. On the other hand, it is not difficult to see that
for every $(k,x)\in\mathcal{M}$  and $n\in\mathbb{N}$ it holds
that
$$
 \lim_{y\to x} \frac{d(f_k^n(x),f_k^n(y))}{d(x,y)}=e^{-S_n\phi(k,x)},
$$
where $S_n\phi$ denotes the $n$-th Birkhoff sum of a function
$\phi: \mathcal{M}\to \mathbb{R}$  given by $\phi(k,x)=\phi_k(x)$.

\begin{claim}
For any $0<\delta\leq \epsilon$ and $0<\lambda<1$, there exists
$K=\exp(C_k\delta^\alpha (1-\lambda^\alpha)^{-1})
>0$  such that for any $(\delta,\lambda)$-hyperbolic pre-ball
$V_k^n(x)$ of order $n$ of a point $(k,x)\in \mathcal{M}$, it holds
$$
K^{-1} e^{-S_n\phi(k,x)} d(y,z) \leq d (f_k^{n}(y),f_k^{n}(z))
\leq K e^{-S_n\phi(k,x)} d(y,z), \qquad \text{for all $y,z\in
\overline{V_k^{n}(x)}$}.
$$
\end{claim}
\begin{proof} By Lemma~\ref{lem:distor} and H\"older assumption of $\phi_n$, we find
$K=\exp(C_k\delta^\alpha (1-\lambda^\alpha)^{-1})>0$ such that
\begin{equation*}
     K^{-1} \leq e^{S_n\phi(k,y)-S_n\phi(k,z)} \leq K,   \quad
     \text{for all $y,z\in \overline{V_k^n(x)}$.}
\end{equation*}
In particular,
$$
   e^{-S_n\phi(k,y)} \leq K e^{-S_n\phi(k,x)}
\quad \text{and} \quad
   e^{S_n\phi(k,y)} \leq K e^{S_n\phi(k,x)}, \quad \text{for all
   $y\in \overline{V_k^n(x)}$}.
$$
This implies that the uniform norms $\|e^{-S_n\phi}\|_{\infty}$
and $\|e^{S_n\phi}\|_{\infty}$ in $V_k^n(x)$ are bounded by $K
e^{-S_n\phi(k,x)}$ and $K e^{S_n\phi(k,x)}$ respectively. Let $y$
and $z$ be a pair of points in the closure of $V_k^n(x)$ and
consider a geodesic $\gamma$ joints them, i.e., a rectificable
curve with length $\ell(\gamma)=d(y,z)$. According to
\cite[Lemma~2.3]{DJ10},
\begin{align}
\label{eq:primera} d(f_k^{n}(y),f_k^{n}(z)) \leq
\|e^{-S_n\phi}\|_{\infty}\, \ell(\gamma) \leq K
e^{-S_n\phi(k,x)}\, d(y,z).
\end{align}
Notice that the inverse map of $f_k^n: V_k^n(x) \to
B(f_k^n(x),\delta)$ is also conformal with pointwise Lipschitz
constant given by the exponential of $S_n\phi(k,y)$. Hence,
arguing similarly, as above, one has that
\begin{align}
\label{eq:segunda} d(y,z)
\leq  K e^{S_n\phi(k,x)}\, d(f_k^n(y),f_k^n(z)), \quad \text{for
all $y,z\in \overline{V_k^n(x)}$}.
\end{align}
Putting together \eqref{eq:primera} and \eqref{eq:segunda}, we
conclude the proof of the claim.
\end{proof}

Now, let $B(x,R)$ and $B(x,r)$ be, respectively, the smallest ball
around $x$ containing $V^{n}_k(x)$ and the largest ball around $x$
contained in $V^{n}_k(x)$. Take $y$ and $z$ in the boundary of
$V_k^n(x)$ so that $d(x,y)=R$ and $d(x,z)=r$.
By the above claim 
\begin{equation}\label{eq-con}
    \delta  K^{-1} e^{S_n\phi(k,x)} \leq r \leq R \leq  K e^{S_n\phi(k,x)} \delta.
\end{equation}
In particular, the ratio of $r$ and $R$ do not depend on $n$.
Equation (\ref{eq-con}) implies that $R \leq  tr$, where
$t=K^2=\exp(2C_k\delta^\alpha (1-\lambda^\alpha)^{-1})$. Since $m$
is locally doubling, being $\delta>0$ small enough (this holds if
$\epsilon>0$ is small) one gets that
$$
\frac{m(B(x,R))}{m(B(x,r))} \leq \frac{m(B(x,tr))}{m(B(x,r))} \leq
L <\infty
$$
and this completes the proof.
\end{proof}

\section{Main results on non-autonomous discrete dynamical systems}
\label{sec:main-result} Now, we give the main results of the
paper. In order to do this we sumarize the assumptions that we
need. We have a non-autonomous discrete  system $f_{1,\infty}$
with $f_{1,\infty}=(f_n)_{n\in\mathbb{N}}$ on a metric measurable
space $(M,d,m)$ or equivalently a skew-product map
$$
   F: \mathbb{N}\times M  \to \mathbb{N}\times M, \qquad F(k,x)=(k+1,f_k(x))
$$
under the following assumptions:
\begin{enumerate}[itemsep=1ex,leftmargin=1cm,label={(H\arabic*)}]
\item \label{H1} \textbf{Hypothesis on metric space:}
$(M,d)$ is a compact metric space. 
\item \label{H2} \textbf{Hypothesis on the fiber maps:} $f_n:M\to M$, for all $n\in\mathbb{N}$, is a
local homeomorphism with uniform Lipschitz constant for the inverse branches. That is, 
for every $n\in \mathbb{N}$, there is a function $\varphi_n:M\to
\mathbb{R}$ such that for each $x \in M$ there exists a
neighborhood $V$ of $x$ in such away that $f_n:V \to f_n(V)$ is invertible
and
\begin{equation*}
d(y,z) \leq \varphi_n(x) \,
d(f_n({y}),f_n(z)), \quad \text{for all for every $y,z \in V$}.
\end{equation*}
%
Additionally, we assume that
$
  \sup\{-\log \varphi_n(x): x\in M, n\in \mathbb{N}\} < \infty.
$
\item  \label{H3} \textbf{Hypothesis on the measure:}
$m$ is a Borel probability on $M$.
We also assume that \\[-0.4cm]
\begin{enumerate}[leftmargin=0.5cm,itemsep=1ex,label={\roman*)}]
\item $m$ is  $f_{1,\infty}$-non-singular, i.e., both $m(f_n(A)) = 0$ and $m(f^{-1}_n(A)) = 0$ whenever $m(A)=0$;
\item $m$ is locally H\"older $f_{1,\infty}$-conformal. That is, there are constants $0<\alpha\leq 1$, $\epsilon>0$ and $C_1>0$ such that  for every $n\in \mathbb{N}$ there is a map $\psi_n: M \to \mathbb{R}$ so that
$$
m(f_n(A)) = \int_{A} e^{-\psi_n(x)} \, dm(x)
$$
for every measurable set $A$ such that $f_n|_A$ is injective and satisfying that
$$
 \qquad  |\psi_n(x)-\psi_n(y)| \leq C_1 \, d(x,y)^\alpha \ \ \text{for all $x,y\in M$ with $d(x,y)<\epsilon$ and $n\in \mathbb{N}$}.
$$
\end{enumerate}
\end{enumerate}

\vspace{0.3cm}

Recalling the notion of local ergodicity in
Definition~\ref{local-ergodicity} we have the following main
result.

\begin{theorem} \label{thm1}
Let $f_{1,\infty}=(f_n)_{n\in\mathbb{N}}$ be a non-autonomous
discrete dynamical system on the metric measure space $(M,d,m)$
under the assumption~\ref{H1}, \ref{H2} and \ref{H3}. Suppose also
that
there is $a\geq 0$ such that
\begin{equation*}
    \limsup_{n\to \infty} \frac{1}{n}\sum_{i=0}^{n-1}
    \log \varphi^{}_{i+1}(f_1^i(x))<-a \quad \text{for $m$-almost every $x\in
    M$}
\end{equation*}
where $f^0_1=\mathrm{id}$ and $f^i_1=f^{}_{i}\circ\dots\circ
f^{}_1$. Then the probability measure $m$ is locally
$f_{1,\infty}$-ergodic if $a=0$ and strong locally
$f_{1,\infty}$-ergodic if $a>0$.
\end{theorem}
\begin{proof}
By assumption $m$ is (strong) $f_{1,\infty}$-expanding for $a=0$
(resp.~$a>0$).  According to Corollary~\ref{cor-expanding}, for
$m$-almost every $x\in M$ we have
  $0<\sigma=\sigma(x)<1$ (resp.~$0<\sigma<1$ uniform on $x$)  such that $(1,x)$ has infinitely
  many $\sigma$-hyperbolic times.
 By Propositions~\ref{prop-delta} and~\ref{prop-distortion},
 there are $0<\delta_1\leq \epsilon$ and $\lambda=\sigma$
 such that $(1,x)$ has infinitely many $(\delta_1,\lambda)$-hyperbolic preballs
 with bounded distortion. Finally by
 Proposition~\ref{prop-ergodic} (resp.~Remark~\ref{rem-ergodic}) we obtain that $m$ is locally (strong) ergodic
 as we want to prove.
\end{proof}

In order to state the second main result we need to impose
slightly strong hypothesis on the measure metric space and the
non-autonomous discrete dynamical system.

\begin{enumerate}[itemsep=1ex,leftmargin=1cm,label={(H\arabic*{*})}]
\item \label{H1'} \textbf{Hypothesis on metric space:}
$(M,d)$ is a compact locally geodesic metric space. 
\item \label{H2'} \textbf{Hypothesis on the fiber maps:}
$f_{1,\infty}$ is locally H\"older conformal. That is, there are
constants $0<\alpha\leq 1$, $\epsilon>0$ and $C_1>0$ such that for
each $n\in\mathbb{N}$ there is a function $\phi_n: M\to
\mathbb{R}$ so that for every $x\in M$,
\begin{equation*}
\lim_{y\to x} \frac{d(f_n(x),f_n(y))}{d(x,y)} = e^{-\phi^{}_n(x)}
\end{equation*}
and
$$
 \qquad  |\phi_n(x)-\phi_n(y)| \leq C_1 \, d(x,y)^\alpha \ \ \text{for all $x,y\in M$ with $d(x,y)<\epsilon$ and $n\in \mathbb{N}$}.
$$
Additionally, we assume that
$$
  \sup\{ -\phi_n(x): x\in M, n\in \mathbb{N}\} < \infty.
$$

\item  \label{H3*} \textbf{Hypothesis on the measure:}
$m$ is a $f_{1,\infty}$-non-singular locally H\"older
$f_{1,\infty}$-conformal Borel probability measure on $M$ as
in~\ref{H3}. We also assume that $m$ is locally doubling, i.e.,
there are $\rho>0$ and a constant $L>0$ such that
$$
m(B(x,2r)) \leq L \, m(B(x,r))
$$
for any ball $B(x, r)$ of radius $0<r \leq \rho$ and $x\in M$.
\end{enumerate}

\vspace{0.2cm} Observe that by setting
$\theta_n(x)=e^{\phi_n(x)}$, according to
\S\ref{sec:local-geodesic-expanding} we have that actually the
maps $f_n$ are local homeomorphisms with uniform Lipschitz
constant $\varphi_n(x)=\|\theta_n\|_{\infty,V}$ at a neighborhood
$V$ of $x$. Thus, hypothesis~\ref{H1'}--\ref{H3*} implies
\ref{H1}--\ref{H3}.

\begin{theorem} \label{thm2}
Let $f_{1,\infty}=(f_n)_{n\in\mathbb{N}}$ be a non-autonomous
discrete dynamical system on the metric measure space $(M,d,m)$
under the assumption~\ref{H1'}, \ref{H2'} and \ref{H3*}. Suppose
also that
there is $a\geq 0$ such that for $m$-almost every $x\in M$ there
is $k=k(x)\in \mathbb{N}$ such that
\begin{equation}\label{eq:thm-expandable}
    \limsup_{n\to \infty} \frac{1}{n}\sum_{i=0}^{n-1}
    \phi^{}_{k+i}(f_k^i(x))<-a
\end{equation}
where $f^0_k=\mathrm{id}$ and $f^i_k=f^{}_{k+i-1}\circ\dots\circ
f^{}_k$. Then the probability measure $m$ is locally
$f_{1,\infty}$-ergodic if $a=0$ and strong locally
$f_{1,\infty}$-ergodic if $a>0$.
\end{theorem}

\begin{proof} From Proposition~\ref{prop-M},~\eqref{eq:thm-expandable} implies that the measure
$m$ is (strong) $f_{1,\infty}$-expandable if $a=0$ (resp.~if
$a>0$). Then, according to Corollary~\ref{cor-expanding}, for
$m$-almost every $x\in M$ we have $k=k(x)\in\mathbb{N}$ and
  $0<\sigma=\sigma(x)<1$ (resp.~$0<\sigma<1$ uniform on $x$)
such that $(k,x)$ has infinitely many $\sigma$-hyperbolic times.
By Propositions~\ref{prop-delta}, \ref{prop-distortion}
and~\ref{prop-conformal}, there are $0<\delta_k\leq \epsilon$ and
$\lambda=\sigma$
 such that $(k,x)$ has infinitely many regular $(\delta_k,\lambda)$-hyperbolic
 preballs with bounded distortion.  Finally by
 Proposition~\ref{prop-ergodic-conformal}
 (resp.~Remark~\ref{rem-ergodic-conformal})
 we obtain that $m$ is locally (strong) ergodic. This completes
 the proof.
\end{proof}

%
%
%

\begin{remark} \label{example} The assumptions \ref{H1'}, \ref{H2} and \ref{H3*} are satisfied if $M$ is a Riemannian compact manifold,
$m$ is the normalized Lebesgue measure of $M$, the fiber maps
$f_n: M \to M$ are $C^{1+\alpha}$ local diffeomorphisms and the
closure of $f_{1,\infty}=(f_n)_{n\in \mathbb{N}}$ is compact in
the space of $C^{1+\alpha}$ local diffeomorphisms of $M$. In this
case,
$$
  \theta_n(x)=\|Df_n(x)^{-1}\| \quad \text{and} \quad  \psi_n(x)=\log |\det Df_n(x)|.
$$
For instance, this is the case when there are $C^{1+\alpha}$ local
diffeomorphisms $g_1,\dots,g_d$ so that $f_n \in
\{g_1,\dots,g_d\}$ for all $n\in\mathbb{N}$ or more general,
$f_{1,\infty}$ is a path in a random walk on
$\mathrm{Diff}^{1+\alpha}(M)$ induced by a probability measure
$\nu$ with compact support. Indeed, since $f_n$ is $C^{1+\alpha}$
and $M$ is compact then $\phi_n=\log \theta_n$ and $\psi_n$ vary
$\alpha$-H\"older continuously with H\"older  constants $C_n>0$
and $H_n>0$ respectively. The compactness of the closure of
$f_{1,\infty}$ implies that $C_n$, $H_n$ and $\|\phi_n\|_\infty$
are, all of them, uniformly bounded. Thus, in order to satisfy
also \ref{H2'} we need to ask that $f_n$ is conformal, i.e.,
$$
\|Df_n(x)^{-1}\|=\|Df_n(x)\|^{-1} \quad \text{for all $x\in M$ and
$n\in\mathbb{N}$.}
$$
%
\end{remark}


\section{Main results on semigroup actions} \label{sec:main-thm-semigroups}
Let $(M,d,m)$ be a compact metric Borel probability space. We
consider a skew-product of the form
 $$
   F:\Omega\times M \to \Omega\times M, \quad F(\omega,x)=(\sigma(\omega), f_{\omega}(x)).
 $$
where the fibers maps $f_\omega:M\to M$ are non-singular with
respect $m$. We have in mind that $\sigma$ is the shift map on
either $\Omega=\mathbb{N}$ or $\Omega=\{1,\dots,d\}^\mathbb{N}$.
In the first case we are modeling a non-autonomous dynamical
systems $f_{1,\infty}=(f_n)_{n\in\mathbb{\mathbb{N}}}$. In the
second case we have the action of a semigroup $\Gamma$ finitely
generated by maps $f_1,\dots,f_d$ so that the fiber mas are
locally constant. That is, $f_\omega=f_i$ if
$\omega=(w_n)_{n\in\mathbb{N}}$ with $w_1=i$. 
Now, we reinterpret in this setting some notions previously
introduced for semigroup action or non-autonomous dynamical
systems.

\subsection{Ergodicity}
We will say that $A\subset M$ is \emph{forward $F$-invariant} set
if $f_\omega(A) \subset A$ for all $\omega\in \Omega$.  A forward
$F$-invariant set $A$ with $m(A) >0$ is called an \emph{ergodic
component} of $m$ with respect to $F$, if it does not admit any
smaller forward $F$-invariant subset with positive $m$-measure.
The measure $m$ is called $F$-ergodic if $M$ is an ergodic
component. Equivalently, if $m(A)\in \{0,1\}$ for all forward
$F$-invariant measurable set $A$ of $M$. Finally, analogously to
Definition~\ref{local-ergodicity} we define \emph{locally (strong)
$F$-ergodicity} in this context.

\begin{proposition}
If $m$ is locally strong $F$-ergodic then $m$ has finitely many
ergodic components.
\end{proposition}

\begin{proof} From the strong $F$-ergodicity we have $\varepsilon>0$
so that for any $F$-invariant set $A$ with positive measure there
is an open ball $B$ of uniform fixed radius with
$m(B)>\varepsilon$ such that $m(B\setminus A)=0$.  Since $M$ is
compact, there can be only finitely many disjoints $F$-invariant
subsets with positive $m$-measure. Hence, we only have finitely
many ergodic components of $m$.
\end{proof}

The formalism of the notion of exactness with respect to $m$,
perviously defined (to the Lebesgue measure), in this context is
the following. We say that $F$ is \emph{$m$-exact} if for every
open set $B$ of $M$, there are sequences $(n_k)_k$ and
$(\omega_k)_k$ in $\mathbb{N}$ and $\Omega$ respectively such that
$$
M = \bigcup_{k\geq 1} f^{n_k}_{\omega_k}(B) \quad \text{modulo a
set of zero $m$-measure.}
$$

\begin{proposition}
\label{prop-ergodic} If $F$ is $m$-exact and $m$ is locally
$F$-ergodic then $m$ is $F$-ergodic.
\end{proposition}
\begin{proof} Let $A$ be a forward $F$-invariant measurable set.
By the local ergodicity of the measure $m$ we get an open set $B$
of $M$ such that $m(B\setminus A)=0$. First observe the following.

\begin{claim} For any function $f$ we have that $f(B)\setminus A
\subset f(B\setminus A)$.
\end{claim}
\begin{proof} If $x\in f(B)\setminus A$ then $x=f(b)\not \in A$ with $b\in
B$. Moreover, $b\not \in A$ since otherwise $f(b)\in f(A)\subset
A$. Thus, $x\in f(B)\setminus A$ as required.
\end{proof}

Now, using this claim and since $F$ is $m$-exact and $A$ is a
forward $F$-invariant set, we get
$$
M \setminus A \subset \bigcup_{k\geq 1}
f^{n_k}_{\omega_k}(B\setminus A) \quad \text{modulo a set of zero
$m$-measure.}
$$
Since $m$ is non-singular, we obtain that $A$ has
full $m$-measure and conclude the proof.
\end{proof}

\subsection{Proof of Theorems~\ref{thmA},~\ref{thmB} and~\ref{thmC}}
Let us consider a semigroup $\Gamma$ finitely generated by
$C^{1+\alpha}$ local diffeomorphisms $f_1,\dots,f_d$ of a compact
manifold $M$. We consider the associated skew-product $F$ as
above. We will first deduce Theorem~\ref{thmA} from
Theorem~\ref{thmB}.

It is not difficult to see that the expansion assumption of
Theorem~\ref{thmA} implies the non-uniform expansion assumption in
Theorem~\ref{thmB}. Thus, we only need to prove that $F$ is exact
(with respect to the Lebesgue measure $m$). This will be achieved
in the following lemma:

 \begin{lemma}
 Assume that there exist $\omega\in \Omega$, $C>0$ and $\lambda>1$  such that
 $$
 \|Df_\omega^n(x)v\|\geq C \lambda^{n} \|v\| \quad \text{for all
 $n\in\mathbb{N}$, $x\in M$ and $v\in T_xM$.}
 $$
 Then, given $x\in M$ and $\varepsilon>0$ there exists $n\in \mathbb{N}$ such that
 $M=f_\omega^{n}(B(x,\varepsilon))$.
 \end{lemma}

 \begin{proof}
 Assume by contradiction that $M\not =f^n_\omega(B)$
 for all $n \in \mathbb{N}$ where $B=B(p,\varepsilon)$
 is the open ball of radius $\varepsilon$ and centered at $x$. Then, for each $n\in\mathbb{N}$ we may a smooth curve $\gamma_n$ joining
 $f^n_\omega(p)$ to a point
 $y_n \in M \setminus f^n_\omega(B)$ of length less than the diameter of the manifold. Since $f^n_\omega$ is a local diffeomorphism, there is a unique curve $\hat{\gamma}_n$ joining $p$ to some point $x \in  M \setminus B$ such that $f^n_\omega(\hat{\gamma}_n) = \gamma_n$. Hence the length of $\gamma_n$ is
 \begin{align*}
  \int \|\gamma'_n(t)\| \, dt =  \int \| Df^n_\omega(\hat{\gamma}_n(t)) \cdot \hat{\gamma}'_n(t) \| \, dt  \geq C\lambda^{n} \int \|\hat{\gamma}'_n(t)\| \, dt.
 \end{align*}
 But since length of $\hat{\gamma}_n$ is larger than $\varepsilon$ we arrive to a contradiction for $n$ large enough.
 \end{proof}

 Theorem~\ref{thmB} immediately follows from
 Remark~\ref{example}, Theorem~\ref{thm1} and
 Proposition~\ref{prop-ergodic}.

Similarly we will prove Theorem~\ref{thmC}. First we need to prove
that if $\Gamma$ is non-uniformly expandable then the Lebesgue
measure $m$ is locally $F$-ergodic. To do this we proceed as in
Theorem~\ref{thm2}. Let $A\subset M$ be a $\Gamma$-invariant set
with $0<m(A)<1$. Since $\Gamma$ is non-uniformly expandable, we
find a Lebesgue density point $x\in A$ and a sequence $\omega\in
\Omega$ such that
\begin{equation*}
   \limsup_{n\to\infty} \frac{1}{n} \sum_{i=0}^{n-1}
  \log \|Df^{}_{\omega_{i+1}}(f^i_\omega(x))^{-1}\|<0.
\end{equation*}
Using Remark~\ref{example}, we have that a non-autonomous
dynamical system $f_{1,\infty}=(f_n)_{n\in\mathbb{N}}$ where
$f_n=f_{\omega_n}$ such that
$$
 \limsup_{n\to \infty} \frac{1}{n}\sum_{i=0}^{n-1} \log
     \theta_{k+i}(f^i_1(x))<0
$$
where $\theta_n(x)=\|Df_n(x)^{-1}\|$ is the pointwise Lipschitz
constant for the inverse branches of $f_n$. According to
Proposition~\ref{prop-M}, there is $a=a(x)>0$ such that $x\in
M(1,a)$. Then, Proposition~\ref{prop-infinite-hyperbolic-times}
implies that there is $\sigma>0$ such that $(1,x)$ has infinitely
many $\sigma$-hyperbolic times. By
Propositions~\ref{prop-delta},~\ref{prop-distortion}
and~\ref{prop-conformal}, there are $0<\delta_1\leq \epsilon$ and
$\lambda=\sigma$ such that $(1,x)$ has infinitely many
$(\delta_1,\lambda)$-hyperbolic regular preballs with bounded
distortion. Finally by Remark~\ref{rem1} we get that there is $z$
such that $m(B(z,\delta/2)\setminus A)=0$. This concludes that $m$
is locally ergodic. Finally, since by assumption, also the action
of $\Gamma$ is exact, then Proposition~\ref{prop-ergodic}
concludes that $m$ is ergodic completing the proof of
Theorem~\ref{thmC}.

\subsection{Examples} We will show some new examples where our main
result Theorem~\ref{thmC} applies. As we indicated in the
introduction,~\cite[Thm.~B]{RZ16} has a gap in its proof and only
works for transitive group of diffeomorphisms. For
\emph{semigroup} action of \emph{local} diffeomorphisms
Theorem~\ref{thmC} requires that the action is exact instance
transitive. From this theorem we cover the result in~\cite{BFMS}
on the ergodicity of the Lebesgue measure for expanding minimal
conformal semigroup action of diffeomorphisms. But also
Theorem~\ref{thmC} extends this result for semigroups of
\emph{local} diffeomorphisms as we will see below. First we
introduce some definitions:

\begin{definition}
The action of a semigroup $\Gamma$ of $C^1$ local diffeomorphisms
of $M$ is said to be backward expanding if there is for every
$x\in M$ there is $h\in \Gamma$ such that $\|Dh(x)^{-1}\|<1$.
\end{definition}

Usually a semigroup action is said to be minimal if every orbit is
dense. Since $M$ is compact, this is equivalent to ask that the
whole space can be covered by finitely many pre-images by elements
of $\Gamma$ of any open set. For this reason we introduce the
following definition:

\begin{definition}
The action of a semigroup $\Gamma$ of local diffeomorphisms of $M$
is said to be backward minimal if for every open set $U\subset M$
there are maps $h_1,\dots, h_n$ in $\Gamma$  such that
$M=h_1(U)\cup \dots \cup h_n(U)$.
\end{definition}

Observe that if the action is backward minimal then it is also
exact. Thus with the above definitions, the following result is a
corollary of Theorem~\ref{thmC}.

\begin{corollary}
Every backward expanding and backward minimal semigroup action of
conformal $C^{1+\alpha}$ local diffeomorphisms of a compact
manifold is ergodic with respect to Lebesgue mesure.
\end{corollary}
\begin{proof}
We only need to note that if the action is backward expanding then
also it is non-uniformly expandable. To do this, we first observe
from the compactness of $M$ and the $C^1$-differentiability of the
maps in $\Gamma$ we get a finite open cover $\{V_1,\dots,V_m\}$ of
$M$ and maps $h_1,\dots, h_m$ in $\Gamma$ such that
$\|Dh_i(x)^{-1}\|<\sigma<1$ for all $x\in V_i$ for all
$i=1,\dots,m$. Thus, given any point $x\in M$ we can construct a
sequence $(i_n)_{n\in \mathbb{N}}$ with $i_n\in \{1,\dots,m\}$
such that $x\in V_{i_1}$ and $h_{i_{n-1}}\circ \dots \circ
h_{i_1}(x)\in V_{i_n}$ for $n\geq 2$. Let $k_n$ be the number of
generators $f_1,\dots,f_d$ involved in the composition of
$h_{i_n}$. Observe that $k_n$ only take finitely many values for
all $n\geq 1$. In particular we have $k\in \mathbb{N}$ such that
$k_n \leq k$ for all $n\in \mathbb{N}$. Take $\omega\in
\Omega=\{1,\dots,d\}^\mathbb{N}$ such that
$f^{\ell_n}_\omega=h_{i_n}\circ \dots \circ h_{i_1}$ where
$\ell_n=k_1+\dots+k_n$ for all $n\in\mathbb{N}$. Hence, by the
conformality of the generators of $\Gamma$ we have
$$
  \frac{1}{\ell_n} \sum_{j=0}^{\ell_n-1}
  \log \|Df^j_{\omega_{j+1}}(f^{j}_\omega(x))^{-1}\| =
  -\frac{1}{\ell_n} \sum_{j=0}^n \log \|Dh_{i_{j+1}}(h_{i_j}(x))\|
  -\leq \frac{n}{\ell_n} \log \sigma \leq -\frac{1}{k} \log \sigma
$$
Now one only need to write $\frac{1}{n} \sum_{i=0}^n a_i
=\frac{\ell_{m}}{n}\frac{1}{\ell_m} \sum_{i=0}^{\ell_m} a_i +
\frac{1}{n} \sum_{i=\ell_m+1}^n a_i$ where $ \ell_m\leq  n \leq
\ell_{m+1}$. Having into account that $\ell_{m+1}-\ell_{m} \leq
k$, $\ell_m \leq k m < k n$ and $a_i=\log
\|Df^i_{\omega_{i+1}}(f^{i}_\omega(x))^{-1}\|$ is uniformly
bounded we conclude~\eqref{eq1}. \end{proof}

Theorem~C can be also used to provide new examples of semigroup
actions of diffeomorphisms which are not expanding as the
following example show.

\begin{example}
Here, we give an example of a semigroup action which is exact and
non-uniformly expanding, but not expanding. Consider the semigroup
$\Gamma$ generated by two $C^{1+\alpha}$ diffeomorphisms $f_0$,
$f_1$ on the unit interval $[0,1]$ with the following properties:
\begin{enumerate}
\item $f_0$ and $f_1$ have both exactly two fixed points: $f_0(0)=f_1(0)=0$ and
$f_0(1)=f_1(1)=1$;
\item $Df_0(0)<1$,  $Df_0(1)=1$ and $Df_1(0)>1$, $Df_1(1)\leq 1$;
\item $\log Df_0(0) / \log Df_1(0) \not \in \mathbb{Q}$;
\item there are points $0<a<c_1<c_2<b<1$ such that
 \begin{enumerate}
 \item $f_{0}([c_1,b])\cup f_1([a,c_2]) \subseteq [a,b]$,
 \item $Df_1(x)>1$ for all $x\in [a,c_1]$ and $Df_0(x)>1$ for all $x\in [c_2,b]$,
 \item $\min_{x\in [c_1,c_2]} \max\{ Df_0(x), Df_1(x)\} >1$.
 \end{enumerate}
\end{enumerate}
Figure \ref{f_0f_1f_2} shows a schematic graph of such
diffeomorphisms. Since both generators have $Df_i(1)\leq 1$ the
action of semigroup $\Gamma$ is not backward expanding on
$M=[0,1]$.
%
%
We claim that the action of semigroup is non-uniformly expanding.
More precisely, we show that for any $x\ne 0,1$, there
$\omega=\omega(x)\in \Omega=\{0,1\}^\mathbb{N}$ with
\begin{equation}\label{eq00}
  \limsup_{n\to\infty} \frac{1}{n} \sum_{i=0}^{n-1}
  \log \|Df^{}_{\omega_{i+1}}(f^i_\omega(x))^{-1}\|<0.
\end{equation}
The conclusion consists of two parts, completely
straightforward.
\begin{itemize}
\item [1)] \label{eq-1} for any $x\in (0,1)$, there is a $m=m(x)$ such that either $f_0^{m}(x)$ or $f_1^{m}(x)$ belongs to $[a,b]$;
\item [2)] \label{eq-2} for any $x\in[a,b]$, there is a sequence
$\bar\omega=(\bar\omega_n)_{n\in\mathbb{N}} \in \Omega$ such that
$$f^n_{\bar\omega}(x)\in [a,b] \quad \text{and} \quad
Df_{\bar\omega_{n+1}}(f^n_{\bar\omega}(x))>1 \quad \text{for any
$n\geq 0$.}
$$
\end{itemize}
Now, for any $x\in (0,1)$, considering the concatenation
$\omega=\omega(x)$ of the words obtaining above we get that
condition~\eqref{eq00} holds along $\omega$. To complete the proof
we need to show that the action of $\Gamma$ is exact. To do this,
first we will observe that it is enough to prove that the orbit by
 the inverse semigroups, i.e., the semigroup
generated by $f_0^{-1}$ and $f_1^{-1}$, of any point in $(0,1)$ is
dense in $M=[0,1]$. Indeed, the density of the backward orbit
provides that for each open set $U$ and point $x\in (0,1)$ we have
a map $h\in \Gamma$ such that $x\in h(U)$.  Since $(0,1)$ is a
Lindel\"of space we can get a countable subcover and thus we get
the action of $\Gamma$ is exact. Now, the density of the backward
orbit of any point $x\in (0,1)$ it follows by the non-resonant
case in~\cite[Lem.~3]{I10} (see also~\cite[Prop.2.1]{GH17}) which
is ours assumption that $f_0$ and $f_1$ has logarithmic rational
independent derivatives at zero.
\begin{figure}
\centering
   \begin{picture}(250,230)
\includegraphics[scale=0.5]{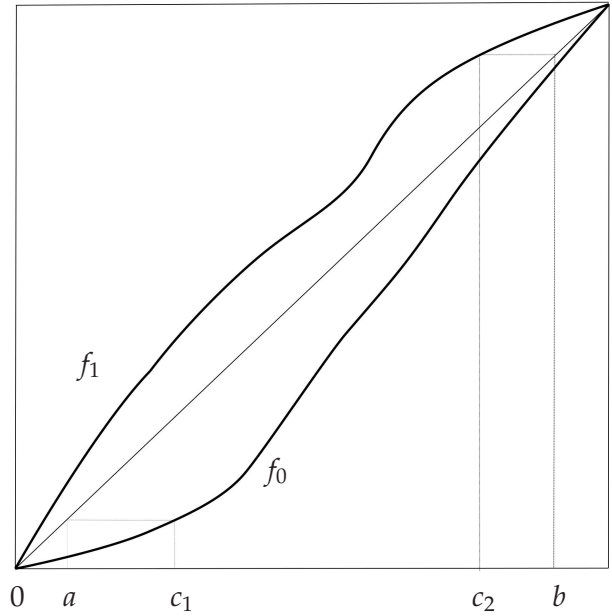}
\put(-200,83){$f_1$} \put(-130,43){$f_0$} \put(-225,-5){$0$}
\put(-205,-5){$a$} \put(-165,-5){$c_1$} \put(-52,-5){$c_2$}
\put(-22,-5){$b$}
 \end{picture}
\caption{Diffeomorphisms $f_0$, $f_1$} \label{f_0f_1f_2}
\end{figure}

\end{example}

Finally, to conclude the paper, we will prove in the following
proposition that there is no finitely generated semigroup action
of diffeomorphisms  in the assumptions of Theorem~\ref{thmB} as we
claimed  in the introduction.
 \begin{proposition}
 There are no non-uniformly expanding  finitely
 generated semigroup actions of diffeomorphisms.
 \end{proposition}
 \begin{proof}
 Suppose that $\Gamma$ is a non-uniformly expanding
 finitely generated semigroup of diffeomorphisms. Hence ,
 there exists $\omega\in\Omega$ such that for $m$-almost
 every $x \in M$ it holds
 $$
    \liminf_{n\to \infty} \frac{1}{n} \sum_{i=0}^{n-1} \log
    \|Df_{\sigma^i(\omega)}(f^i_\omega(x))^{-1}\|^{-1}> 0.
 $$
 Since $\|T^{-1}\|^{-1}\leq |\det T|^{1/s}$ for all linear operator $T$ on a $s$-dimensional vector space, one has that
 $$
  \liminf_{n\to \infty} \frac{1}{n} \log |\det Df^n_\omega(x)|  =
  \liminf_{n\to \infty} \frac{1}{n} \sum_{i=0}^{n-1} \log
    |\det Df_{\sigma^i(\omega)}(f^i_\omega(x))|> 0.
 $$
 Since $f_i$ is a diffeomorphisms for all $i=1,\dots,d$, changing variables we have
 that
 $$
   \int  |\det  Df^n_\omega(x)| \, dm(x)=1.
 $$
 Hence, by Fatou-Lebesgue lemma since $|\det Df_i|$ is uniformly bounded for
 all $i=1,\dots,d$ and using the Jensen inequality we get that
 \begin{align*}
      0 &=  \liminf_{n\to\infty} \frac{1}{n} \log \int  |\det
      Df^n_\omega(x)| \, dm(x) \geq \liminf_{n\to\infty} \frac{1}{n} \int  \log |\det
      Df^n_\omega(x)| \, dm(x) \\ &\geq \int \liminf_{n\to\infty} \log |\det
      Df^n_\omega(x)| \, dm(x) > 0.
 \end{align*}
 This provides a contradiction and concludes the proof of the
 proposition.
 \end{proof}

\subsection*{Acknowledgements}
During the preparation of this article the first author was
supported by MTM2014-56953-P and CNPQ-Brazil and the second author by IPM (No. 96370119).
The second author also thanks ICTP for supporting through the association schedule.

\end{document}